\documentclass[12pt]{article}
\usepackage{amsthm}
\usepackage{amsfonts}
\usepackage{amssymb}
\usepackage{amsgen}
\usepackage{amsmath}
\usepackage{amsopn}
\usepackage{verbatim}
\usepackage{xypic}
\usepackage{xspace}
\usepackage{multicol}
\usepackage{makeidx}
\usepackage{eepic}
\usepackage{epsf}
\usepackage{upref}
\usepackage[mathscr]{eucal}

\theoremstyle{plain}
\newtheorem{thm}{Theorem}[section]

\newtheorem{prop}[thm]{Proposition}
\newtheorem{cor}[thm]{Corollary}

\theoremstyle{definition}

\newcommand{\Z}{{\mathbb Z}}

\newcommand{\N}{{\mathbb N}}

\DeclareMathOperator{\sh}{{Shfl}}

\newcommand{\gd}{{\delta}}

\newcommand{\gs}{{\sigma}}

\newcommand{\gz}{{\zeta}}

\renewcommand{\N}{{\mathbb  N}}

\newcommand{\ZN}{{\mathbb  Z_{\ge 0}}}
\newcommand{\Hq}{{H_q}}
\newcommand{\sHq}{{H_q^*}}
\newcommand{\tHq}{{\tilde{H}_q}}

\newcommand{\ors}{{\mathbf{s}}}
\newcommand{\ort}{{\mathbf{t}}}

\begin{document}

\title{On $q$-Analog of Wolstenholme Type Congruences for Multiple Harmonic Sums}
\author{Jianqiang Zhao}
\date{}
\maketitle

\begin{center}
 {Department of Mathematics, Eckerd College, St. Petersburg, FL 33711}
\end{center}

\textbf{Abstract.} Multiple harmonic sums are iterated generalizations of harmonic sums. Recently Dilcher has considered congruences involving $q$-analogs of these sums in depth one. In this paper we shall study the homogeneous case for arbitrary depth by using generating functions and shuffle relations of the $q$-analog of multiple harmonic sums. At the end, we also consider some non-homogeneous cases.


\section{Introduction.}

In \cite{SP} Shi and Pan extended Andrews' result \cite{An} on the
$q$-analog of Wolstenholme Theorem to the following two cases: for all prime $p\ge 5$
\begin{alignat}{2}\label{equ:SP1}
  \sum_{j=1}^{p-1} \frac{1}{[j]_q} \equiv & \frac{p-1}2 (1-q) +\frac{p^2-1}{24} (1-q)^2 [p]_q &\pmod{[p]_q^2}, \\
 \sum_{j=1}^{p-1} \frac{1}{[j]_q^2} \equiv & - \frac{(p-1)(p-5)}{12} (1-q)^2    &\pmod{[p]_q},  \label{equ:SP2}
\end{alignat}
where $[n]_q=(1-q^n)/(1-q)$ for any $n\in \N$ and $q\ne 1$. This type of congruences is considered
in the polynomial ring $\Z[q]$ throughout this paper. Notice that the modulus $[p]_q$ is an irreducible polynomial in $q$ when $p$ is a prime. In \cite{D} Dilcher
generalized the above two congruences further
to sums of the form  $\sum_{j=1}^{p-1} \frac{1}{[j]_q^n}$
and $\sum_{j=1}^{p-1} \frac{q^n}{[j]_q^n}$ for
all positive integers $n$ in terms of certain determinants of binomial coefficients.
However, his modulus is always $[p]_q$.
He also expressed these congruences using Bernoulli numbers, Bernoulli numbers of the second kind,
and Stirling numbers of the first kind, which we briefly recall now.

The well-known Bernoulli numbers are defined by the following generating series:
$$\frac{x}{e^x-1}=\sum_{n=0}^\infty B_n \frac{x^n}{n!}=1-\frac{1}{2}\frac{x}{1!}+\frac{1}{6}\frac{x^2}{2!}-\frac{1}{30}\frac{x^4}{4!}+\cdots.$$
On the other hand, the Bernoulli numbers of the second kind
are defined by the power series (cf. \cite[p.\ 114]{roman}).
$$\frac{x}{\log(1+x)}=\sum_{n=0}^\infty b_n \frac{x^n}{n!}=1+\frac{1}{2}\frac{x}{1!}-\frac{1}{6}\frac{x^2}{2!}+\frac{1}{4}\frac{x^3}{3!}-\frac{19}{24}\frac{x^4}{4!}+\cdots.$$
This is a little different from the definition
of $\tilde b_n$ in \cite{D}, which is changed to $b_n$ later in the same paper.
Finally, the Stirling numbers of the first kind $s(n, j)$ are defined by
$$ x(x-1)(x-2)\cdots(x-n+1)= \sum_{j=0}^n s(n,j) x^j.$$
Define
\begin{equation}\label{equ:Cnp}
 K_n(p):=(-1)^{n-1}\frac{b_n}{n!}-\frac{(-1)^n}{(n-1)!}\sum_{j=1}^{[n/2]}
\frac{B_{2j}}{2j} s(n-1, 2j-1)p^{2j} .
\end{equation}
By \cite[Thm.\ 1, (6.5) and Thm.\ 4]{D} and \cite[Thm.\ 3.1]{How} one gets:
\begin{thm} \label{thm:DilcherThm1}
If $p> 3$ is a prime, then for all integers $n>1$ we have
$$\sum_{j=1}^{p-1} \frac{q^j}{[j]_q^n}\equiv K_n(p) (1-q)^n    \pmod{[p]_q}.
$$
\end{thm}
We will need the following easy generalization of this theorem.
\begin{thm} \label{thm:DilcherThm2}
If $p> 3$ is a prime, then for all integers $n>t\ge 1$ we have
\begin{equation}\label{equ:genmod}
 \sum_{j=1}^{p-1} \frac{q^{tj}}{[j]_q^n}\equiv (1-q)^n \sum_{i=0}^{t-1} {t-1\choose i}(-1)^i K_{n-i}(p)   \pmod{[p]_q}.
\end{equation}
Moreover,
\begin{equation}\label{equ:t=0}
 \sum_{j=1}^{p-1} \frac{1}{[j]_q^n}\equiv (1-q)^n \left(\frac{p-1}2+\sum_{j=2}^{n}  K_{j}(p) \right)  \pmod{[p]_q}.
\end{equation}
\end{thm}
\begin{proof} If $t>1$ it is clear that
$$q^{tj}=q^j\big (1-(1-q^j)\big)^{t-1}=q^j\sum_{i=0}^{t-1} {t-1\choose i}(-1)^i  (1-q^j)^i.$$
So \eqref{equ:genmod} follows from Theorem \ref{thm:DilcherThm1} immediately. Congruence \eqref{equ:t=0} is a variation of \cite[(5.11)]{D}.
\end{proof}

All of the sums in Theorem \ref{thm:DilcherThm1} and \ref{thm:DilcherThm2} are special cases
of the $q$-analog of multiple harmonic sums.
The congruence properties of the classical multiple harmonic sums (MHS for short)
are systematically investigated in \cite{wols}. In this paper we shall study their
$q$-analogs which are natural generalizations
of the congruences obtained by Shi and Pan \cite{SP} and Dilcher \cite{D}.

Similar to its classical case (compare \cite{wols})
a $q$-analog of multiple harmonic sum ($q$-MHS for short) is defined as follows.
For $\ors:=(s_1,\dots, s_\ell)\in \N^\ell$, $\ort:=(t_1,\dots, t_\ell)\in \N^\ell$ and $n\in \ZN$ set
\begin{equation}\label{equ:defnH}
H_q^{(\ort)}(\ors;n):=\sum_{1\le k_1<\dots<k_\ell\le n}
\frac{q^{k_1t_1+\cdots +k_\ell t_\ell}}{[k_1]_q^{s_1}\cdots [k_\ell]_q^{s_\ell}}, \ \quad \sHq^{(\ort)}(\ors;n)=H_q^{(\ort)}(\ors;n)/(1-q)^{w(\ors)} ,
\end{equation}
where $w(\ors):=s_1+\cdots+s_\ell$ is the \emph{weight}, $\ell$ the \emph{depth} and $\ort$ the \emph{modifier}.
For trivial modifier we set
$$\Hq(\ors;n):=H_q^{(0,\dots,0)}(\ors;n), \ \qquad \sHq(\ors;n)=\Hq(\ors;n)/(1-q)^{w(\ors)}. $$
Note that in \cite{D} $\tHq(s;p-1):=H_q^{(1)}(s;p-1)$ are studied in some detail and are related to $\Hq(s;p-1)$.
Also note that  $H_q^{(s_1-1,\dots,s_\ell-1)}(\ors;n)$ are the partial sums of the most convenient
form of $q$-multiple zeta functions (see \cite{qmz}).

In this paper we mainly consider $q$-MHS with the trivial modifier.
By convention we set $H_q^{(\ort)}(\ors;r)=0$ for $r=0,\dots,\ell-1$, and
$H_q^{(\ort)}(\emptyset;n)=1$.
To save space, for an ordered set $(e_1,\dots,e_t)$ we denote by
$\{e_1,\dots, e_t\}^d$ the ordered set formed by repeating
$(e_1,\dots,e_t)$ $d$ times. For example $\Hq(\{s\}^\ell;n)$ will be called
a {\em homogeneous} sum.

Throughout the paper, we use short-hand $\Hq(\ors)$ to denote $\Hq(\ors;p-1)$ for some fixed prime $p$.

\section{Homogeneous $q$-MHS.}

It is extremely beneficial to study the so-called stuffle (or quasi-shuffle) relations
among MHS (see, for e.g., \cite{wols}). The same mechanism works equally well for $q$-MHS.

Recall that for any two ordered sets $(r_1,\dots,r_t)$ and
$(r_{t+1},\dots,r_n)$ the shuffle operation is defined by
$$\sh\big((r_1,\dots,r_t), (r_{t+1},\dots,r_n)\big):=
\bigcup_{\substack{\gs\text{ permutes } \{1,\dots,n\}, \\
\gs^{-1}(1)<\cdots<\gs^{-1}(t),\\
\gs^{-1}(t+1)<\cdots<\gs^{-1}(n)}}
 \big(r_{\gs(1)},\dots,r_{\gs(n)}\big).$$
Fix a positive integer $s$. For any $k=1,\dots, \ell-1,$ we have by stuffle relation
$$\sHq\big((\ell-k)s\big) \cdot \sHq\big(\{s\}^{k}\big)=
\sum_{\ors\in \sh\big(\{(\ell-k)s\},\{s\}^{k}\big)} \sHq(\ors )+
\sum_{\ors\in\sh\big(\{(\ell-k+1)s\},\{s\}^{k-1}\big)}\sHq(\ors).
$$
Applying $\sum_{k=1}^{\ell-1} (-1)^{\ell-k-1}$ on both sides we get
\begin{equation} \label{equ:homoShuffle}
\sHq\big(\{s\}^\ell\big)=\frac1\ell \sum_{k=0}^{\ell-1}  (-1)^{\ell-k-1} \sHq\big((\ell-k)s\big) \cdot
\sHq\big(\{s\}^{k}\big).
\end{equation}

\begin{thm} \label{thm:homoAll}
Let $s$ be a positive integer and let $\eta_s=\exp(2\pi i/s)$
be the $s$th primitive root of unity. Then
$$\sum_{\ell=0}^\infty \sHq\big(\{s\}^\ell \big) x^\ell
\equiv \frac{(-1)^s}{p^s x}\prod_{n=0}^{s-1}  \Big(1- (1-\eta_s^n (-x)^{1/s} )^p\Big) \pmod{[p]_q}.$$
\end{thm}
\begin{proof}
Let $\gz=\exp(2\pi i/p)$ be the primitive $p$th root of unity and set
\begin{equation}\label{equ:Pk}
P_n= \sum_{j=1}^{p-1} \frac{1}{(1-\gz^j )^n}.
\end{equation}
It is easy to see that $ \sHq(n)\equiv P_n\pmod{[p]_q}$.
By using partial fractions Dilcher \cite[(4.2)]{How} obtained essentially the following generating function of $P_n$:
\begin{equation}\label{equ:gen}
 g(x):=\sum_{n=0}^\infty P_n x^n =- \frac{px(x-1)^{p-1} } {1- (1-x)^p}.
\end{equation}
Let $a_\ell=\sHq\big(\{s\}^\ell\big)$ for all $\ell\ge 0$.
Let $w(x)=\sum_{\ell=0}^\infty a_\ell x^\ell$ be its the generating function.
By \eqref{equ:homoShuffle} we get
$$ w(x)=\sum_{\ell=0}^\infty a_\ell x^\ell
 \equiv 1+\sum_{\ell=1}^\infty  \frac{1}{\ell} \sum_{k=0}^{\ell-1} (-1)^{\ell-k-1} P_{(\ell-k)s} a_k x^\ell  \pmod{[p]_q}.$$
Differentiating both sides and changing index $\ell\to\ell+1$ we get modulo  $[p]_q$
\begin{align*}
 w'(x)\equiv &\sum_{\ell=0}^\infty   \sum_{k=0}^{\ell} (-1)^{\ell-k} P_{(\ell-k+1)s} a_k x^\ell \\
\equiv&  \sum_{k=0}^\infty \sum_{\ell=k}^\infty (-1)^{\ell-k} P_{(\ell-k+1)s} a_k x^\ell \\
\equiv&  w(x) \sum_{\ell=0}^\infty P_{(\ell+1)s} (-x)^{\ell}     \\
\equiv&   \frac{w(x)}{-x}   \left(\sum_{\ell=0}^\infty P_{ \ell s} (-x)^{\ell}+1\right)     \\
\equiv&  \frac{w(x)}{-sx} \left(s+ \sum_{n=0}^{s-1}
  \sum_{\ell=0}^\infty P_\ell  (\eta_s^n (-x)^{1/s})^{\ell} \right)    \\
\equiv&  \frac{w(x)}{-sx} \Big(s+  \sum_{n=0}^{s-1} g\big(\eta_s^n (-x)^{1/s} \big)  \Big)  \\
\equiv&  \frac{w(x)}{-sx} \left( s -\sum_{n=0}^{s-1}
\frac{p\eta^n (-x)^{1/s} (\eta^n(-x)^{1/s} -1)^{p-1} } {1- (1-\eta_s^n (-x)^{1/s} )^p} \right) .
\end{align*}
Here $\eta_s=\exp(2\pi i/s)$ is the $s$th primitive root of unity.
Thus
$$(\ln w(x))'= \left( -(\ln x)'+\sum_{n=0}^{s-1}
\frac{(1- (1-\eta^n (-x)^{1/s} )^p)' } {1- (1-\eta_s^n (-x)^{1/s} )^p} \right).$$
Therefore by comparing the constant term we get
$$w(x)\equiv \frac{(-1)^s}{p^s x}\prod_{n=0}^{s-1}  \Big(1- (1-\eta_s^n (-x)^{1/s} )^p\Big) \pmod{[p]_q}$$
as desired.
\end{proof}

\begin{cor} \label{cor:homo1}
For all positive integer $\ell<p$ we have
$$\Hq\big(\{1\}^\ell\big)\equiv \frac1{\ell+1} {p-1\choose \ell} \cdot(1-q)^{\ell} \qquad \pmod{[p]_q}.$$
\end{cor}
\begin{proof} By the theorem we get
\begin{align*}
\sum_{\ell=0}^\infty \sHq\big(\{1\}^\ell \big) x^\ell
\equiv & \frac{(1+x)^p-1}{px}   \\
\equiv & \frac1{px} \sum_{\ell=0}^\infty  {p\choose \ell+1}  x^{\ell+1}
\equiv\sum_{\ell=0}^\infty \frac1{\ell+1} {p-1\choose \ell}  x^\ell \pmod{[p]_q}.
\end{align*}
The corollary follows immediately.
\end{proof}

\begin{cor} \label{cor:homo2}
For every positive integer $\ell<p$ we have
\begin{equation*}
 \Hq\big(\{2\}^\ell\big)\equiv (-1)^\ell \frac{2\cdot \ell!}{(2\ell+2)!}
   {p-1\choose \ell}\cdot F_{2,\ell}(p) \cdot (1-q)^{2\ell} \pmod{[p]_q},
\end{equation*}
where $F_{2,\ell}(p)$ is a monic polynomial in $p$ of degree $\ell$.
\end{cor}
\begin{proof}
By Theorem \ref{thm:homoAll} we have modulo $[p]_q$
\begin{align*}
 \sum_{\ell=0}^\infty \sHq\big(\{2\}^\ell \big) x^\ell
\equiv& \frac{1}{p^2 x} \Big(1- (1- i\sqrt{x} )^p\Big)\Big(1- (1+i\sqrt{x} )^p\Big)\\
\equiv& \frac{1}{p^2 x} \left|  \sum_{j=1}^{(p-1)/2} {p\choose 2j}(-1)^j x^{j}
+i\sqrt{x}  \sum_{j=0}^{(p-1)/2} {p\choose 2j+1}(-1)^j x^{j}
\right|^2,
\end{align*}
which easily yields
$$ \sHq\big(\{2\}^\ell\big)\equiv  \frac{(-1)^\ell}{p^2} \left\{
  \sum_{\substack{j+k=\ell\\ 0\le j,k< p/2}} {p\choose 2j+1}{p\choose 2k+1}
  -  \sum_{\substack{j+k=\ell+1 \\ 1\le j,k< p/2}} {p\choose 2j}{p\choose 2k}
 \right\} . $$
In the first sum above
if $j+k=\ell+1$ and $1\le j,k< p/2$ then
we may assume $j>\ell/2$. Then $(\ell+1)! {p\choose \ell+1}$ is a factor
of $(2j+1)! {p\choose 2j+1}$ as a polynomial of $p$, so is $\ell! {p-1\choose \ell}$.
Similarly we can see that $\ell! {p-1\choose \ell}$ is a factor of the second sum.

In order to determine the leading coefficient we set
 \begin{align*}
 C_1(x)=&\sum_{j=0}^\ell \frac{(2\ell+2)!\, x^{2j+1}}{(2j+1)!(2l-2j+1)!}=
\frac{(x+1)^{2\ell+2}-(x-1)^{2\ell+2}}{2 },\\
 C_2(x)=&\sum_{j=0}^{\ell+1} \frac{(2\ell+2)!\, x^{2j}}{(2j)!(2l-2j+2)!}=
\frac{(x+1)^{2\ell+2}+(x-1)^{2\ell+2}}{2 }.
 \end{align*}
Hence
 \begin{align*}
  \ &\sum_{\substack{j+k=\ell\\ 0\le j,k< p/2}} \frac1{(2j+1)!(2k+1)!}
  -  \sum_{\substack{j+k=\ell+1 \\ 1\le j,k< p/2}} \frac1{(2j)!(2k)!}\\
  =& \frac{C_1(1)- (C_2(1)- 2)}{(2\ell+2)!} =\frac{2}{(2\ell+2)!}.
 \end{align*}
This finishes the proof of the corollary.
\end{proof}

\begin{cor} \label{cor:homo33}
Let $\ell$ be a positive integer. Set $\gd_\ell=(1+(-1)^\ell)$ and $L=3\ell+3$. Then for every
prime $p\ge L$ we have modulo $[p]_q$
\begin{equation}  \label{equ:homo333}
 \Hq\big(\{3\}^\ell\big)
\equiv \left\{
          \begin{array}{ll}
   \displaystyle        \frac{- 3 \cdot \ell!} {(3\ell+1)!}
   {p-1\choose \ell}\cdot F_{3,\ell}(p)\cdot (1-q)^{3\ell} \phantom{\sum_{a}} , & \hbox{if $\ell$ is odd,} \\
  \displaystyle         \frac{6\cdot \ell!}{(3\ell+3)!}
   {p-1\choose \ell}\cdot F_{3,\ell}(p)\cdot (1-q)^{3\ell} \phantom{\sum^{a}} , & \hbox{if $\ell$ is even,}
          \end{array}
        \right.
\end{equation}
where $F_{3,\ell}(p)$ is a monic polynomial in $p$ of degree $2\ell-1$
if $\ell$ is odd and of degree $2\ell$ if $\ell$ is even.
\end{cor}
\begin{proof}
Let $\eta=\exp(2\pi i/3)$. Then $\eta^2+\eta+1=0$. By Theorem \ref{thm:homoAll} we have
\begin{equation}\label{equ:homo3step1}
 \sum_{\ell=0}^\infty \sHq\big(\{3\}^\ell \big) x^\ell
\equiv \frac{-1}{p^3 x} \prod_{a=0}^2 \Big(1-(1-\eta^a\sqrt[3]{-x} )^p\Big).
\end{equation}
We now use two ways to expand this. Set $y=\sqrt[3]{-x}$. First, the product on the right hand side of \eqref{equ:homo3step1}
can be expressed as
\begin{align*}
& 1-\sum_{a=0}^2 (1-\eta^ay )^p +\sum_{a=0}^2 (1-\eta^ay )^p (1-\eta^{a+1}y )^p
-\prod_{a=0}^2 (1-\eta^ay )^p\\
=&  1- \sum_{j=0}^{p} {p\choose j} \sum_{a=0}^2 \eta^{aj} y^j+
\sum_{a=0}^2 (1+\eta^ay +\eta^{a+1} y^2 )^p
-(1+x)^p \\
=&  1-3\sum_{j=0}^{[p/3]} {p\choose 3j}x^j+
 3 \sum_{\substack{j,k\ge 0,j+k<p \\ 2j+k \equiv 0 (3) }}  \frac{p! \, (-x)^{(j+2k)/3}}{j!k!(p-j-k)!}
-(1+x)^p .
\end{align*}
Thus for $\ell>0$ we get
$$
 \sHq\big(\{3\}^\ell\big)\equiv  \frac{1}{p^3} \left\{
 3\gd_\ell {p\choose L} +(-1)^{\ell} \cdot 3
  \sum_{k\ge 1} {p\choose L-k}{L-k \choose k}
+ {p\choose \ell+1}
 \right\}
$$
Note that if $\ell$ is odd then the degree
of the polynomial is reduced to $3\ell-1$ with leading coefficient given by
$$(-1)^{\ell} \cdot 3 \frac{1}{(L-1)!}{L-1 \choose 1}=\frac{-3}{(L-2)!}=\frac{-3}{(3\ell+1)!}$$
as we wanted.

Now to prove $\ell!{p\choose \ell}$ is a factor we use the following expansion of
\eqref{equ:homo3step1}:
$$  \sum_{\ell=0}^\infty
\frac{1}{p^3 x} \sum_{ j,k,n \ge 1 }
(-1)^{j+k+n} {p\choose j}{p\choose k} {p\choose n} x^{(j+k+n)/3} \eta^{k+2n}.$$
Thus
$$\sHq\big(\{3\}^\ell \big) \equiv
\frac{1}{p^3} \sum_{\substack{1\le j,k,n\le p\\ j+k+n=3\ell+3}}
(-1)^{j+k+n} {p\choose j}{p\choose k} {p\choose n} \eta^{k+2n} \quad \pmod{[p]_q}.$$
Notice that $j+k+n=3\ell+3$ implies one of the indices, say $j$, is at least $\ell+1$.
Then clearly ${p\choose j}$ contains $\ell!{p\choose \ell}$ as a factor, therefore so
does $\sHq\big(\{3\}^\ell \big) \pmod{[p]_q}$. This completes the proof of the corollary.
\end{proof}

\section{Some non-homogeneous $q$-MHS congruences.}

In this section we consider some non-homogeneous $q$-MHS of depth two with modifiers of special type.

\begin{thm}
Let $m,n$ be two positive integers. For every prime $p$ we have
$$
H_q^{(m,n)}(2m,2n)\equiv \frac12\left\{f(m;p)f(n;p)-f(m+n;p) \right\} \pmod{[p]_q}.
$$
where
$$f(N;p)=(1-q)^{2N}\sum_{i=0}^{N-1} {N-1\choose i}(-1)^i K_{2N-i}(p)$$
\end{thm}
\begin{proof}
By definition and substitution $i\to p-i$ and $j\to p-j$ we have
\begin{alignat}{2}
 \sHq^{(m,n)}(2m,2n)=& \sum_{1\le i<j<p} \frac{q^{mi+nj}}{(1-q^i)^{2m}(1-q^j)^{2n}}  \notag \\
=& \sum_{1\le j<i<p} \frac{q^{pm+pn-mi-nj}}{(1-q^{p-i})^{2m}(1-q^{p-j})^{2n}} \notag \\
\equiv & \sum_{1\le j<i<p} \frac{q^{mi+nj}}{(q^i-q^{p})^{2m}(q^j-q^{p})^{2n}}  \qquad &\pmod{[p]_q} \notag\\
\equiv & \sum_{1\le j<i<p} \frac{q^{mi+nj}}{(1-p^i)^{2m}(1-p^j)^{2n}}  &\pmod{[p]_q} \notag \\
\equiv & \sHq^{(n,m)}(2n,2m)       &\pmod{[p]_q} \label{equ:rev}
\end{alignat}
By shuffle relation we have
$$ \sHq^{(m)}(2m) \sHq^{(n)}(2n)=\sHq^{(m,n)}(2m,2n)+\sHq^{(n,m)}(2n,2m)+\sHq^{(m+n)}(2m+2n).
$$
Together with \eqref{equ:rev} this yields
$$ 2\sHq^{(m,n)}(2m,2n)\equiv \sHq^{(m)}(2m) \sHq^{(n)}(2n)-\sHq^{(m+n)}(2m+2n) \pmod{[p]_q}.$$
Our theorem follows from \eqref{equ:genmod} quickly.
\end{proof}

In the study of $q$-multiple zeta functions the following function appears naturally
(see \cite[(47)]{qmz} or \cite[Theorem 1]{B}):
$$\varphi_q(n)= \sum_{k=1}^\infty (k-1)\frac{q^{(n-1)k}}{[k]_q^n}=
\sum_{k=1}^\infty \frac{kq^{(n-1)k}}{[k]_q^n}-\zeta_q(n),$$
where $\zeta_q(n)=\sum_{k=1}^\infty \frac{q^{(n-1)k}}{[k]_q^n}$ is the $q$-Riemann zeta
value defined by Kaneko et al. in \cite{KKW}. Using the results we have
obtained so far in this paper we discover  a congruence
related to the partial sums of $\varphi_q(2)$.
\begin{prop}
For every prime $p$ we have
$$
\sum_{k=1}^{p-1} \frac{ k q^{k}}{[k]_q^2} \equiv -\frac{p(p-1)(p+1)}{24} (1-q)^2 \pmod{[p]_q}.$$
\end{prop}
\begin{proof} We can check the congruence for $p=2$ and $p=3$  easily by hand. Now we assume $p\ge 5$.
By definition we have
$$\sHq(2,1)=\sum_{1\le i<j<p} \frac{1}{(1-q^i)^2(1-q^j)}.$$
With substitution $i\to p-i$ and $j\to p-j$ we get modulo $[p]_q$
\begin{align*}
-\sHq(2,1)=&-\sum_{1\le j<i<p} \frac{q^{2i}\cdot q^j}{(q^i-q^p)^2(q^j-q^p)}  \\
\equiv &-\sum_{1\le j<i<p} \frac{q^{2i}\cdot q^j}{(q^i-1)^2(q^j-1)}   \\
\equiv &-\sum_{1\le j<i<p} \frac{(q^i-1)^2+2(q^i-1)+1}{(q^i-1)^2} \cdot \frac{1-q^j-1}{1-q^j}  \\
\equiv & \sHq(1,2)-2\sHq(1,1)+\sum_{k=1}^{p-1}  \frac{p-3+k}{1-q^k}
 -\sum_{k=1}^{p-1}  \frac{k-1}{(1-q^k)^2}-{p-1\choose 2} \\
\equiv & \sHq(1,2)-2\sHq(1,1)+ (p-3)\sHq(1)+ \sHq(2)-{p-1\choose 2} - \sum_{k=1}^{p-1} \frac{ k q^{k}}{(1-q^k)^2} .
\end{align*}
Notice that we have the stuffle relations
$$\sHq(2,1)+\sHq(1,2)=\sHq(1)\sHq(2)-\sHq(3),\quad 2\sHq(1,1)=\sHq(1)^2-\sHq(2).$$
Hence modulo $[p]_q$
$$ \sum_{k=1}^{p-1} \frac{ k q^{k}}{(1-q^k)^2} \equiv  (\sHq(1)+2)\sHq(2)-\sHq(3) -\sHq(1)^2 + (p-3)\sHq(1) -{p-1\choose 2}.$$
Notice that by \cite[Theorem 2]{D}
\begin{equation}\label{equ:H3}
\sHq(3)  \equiv  -\frac{(p-1)(p-3)}{8}    \pmod{[p]_q}.
\end{equation}
The proposition now follows from \eqref{equ:SP1} and \eqref{equ:SP2} immediately.
\end{proof}

\section{A congruence of Lehmer type}

Instead of the harmonic sums up to $(p-1)$-st term Lehmer also studied the following type of congruence (see
\cite{L}): for every odd prime $p$
$$\sum_{j=1}^{(p-1)/2} \frac 1j \equiv -2q_p(2)+q_p(2)^2 p\pmod{p^2},$$
where $q_p(2)=(2^{p-1}-1)/p$ is the Fermat quotient. It is also easy to see that for every
positive integer $n$ and prime $p>2n+1$
$$\sum_{j=1}^{(p-1)/2} \frac 1{j^{2n}} \equiv 0 \pmod{p}.$$
As a $q$-analog of the above we have
\begin{thm}
Let $n$ be a positive integer. For every odd prime $p$ we have
$$ H_q^{(n)}(2n;(p-1)/2)
 \equiv \frac12 (1-q)^{2n} \sum_{j=0}^{n-1} {n-1\choose j}(-1)^j K_{2n-j}(p)   \quad \pmod{[p]_q}.$$
\end{thm}
\begin{proof}
By definition and substitution $i\to p-i$ we have
\begin{align*}
 \sHq^{(n)}(2n)=& \sHq^{(n)}(2n;(p-1)/2) +\sum_{1\le i\le (p-1)/2}  \frac{q^{n(p-i)}}{(1-q^{p-i})^{2n}} \\
\equiv  & 2 \sHq^{(n)}(2n;(p-1)/2) \pmod{[p]_q}\\
\end{align*}
By \eqref{equ:genmod} this yields the theorem quickly.
\end{proof}

To conclude the paper we remark that the congruence for general $q$-MHS should involve some type of $q$-analog of Bernoulli numbers and Euler numbers similar to the classical cases treated in \cite{wols}. We hope to return to this theme in the future.

\medskip

\noindent{\bf Acknowledgement.} This work is partly supported by NSF grant DMS1162116.

\end{document}